\documentclass[reqno,final]{amsart}
\usepackage{natbib}  
\usepackage{fancyhdr} 
\usepackage{color} 
\usepackage{hyperref} 
\usepackage{graphicx} 

\definecolor{aleacolor}{rgb}{0.16,0.59,0.78}

\hypersetup{
breaklinks,
colorlinks=true,
linkcolor=aleacolor,
urlcolor=aleacolor,
citecolor=aleacolor}

\renewcommand{\headheight}{11pt}
\renewcommand{\cite}{\citet}

\theoremstyle{plain}
\newtheorem{theorem}{Theorem}[section]                                          
                          
\newtheorem{lemma}[theorem]{Lemma}

\theoremstyle{definition}
\newtheorem{definition}[theorem]{Definition}
\theoremstyle{remark}
\newtheorem{remark}[theorem]{Remark}

\makeatletter \@addtoreset{equation}{section} \makeatother
\renewcommand\theequation{\thesection.\arabic{equation}}
\renewcommand\thefigure{\thesection.\arabic{figure}}
\renewcommand\thetable{\thesection.\arabic{table}}

\usepackage{bm}
\usepackage{amssymb}
\newtheorem{appxlem}{Lemma}[section]
\newtheorem{appxcor}{Corollary}[section]
\newcommand{\N}{\mathbb{Z}_{+}}
\newcommand{\Z}{\mathbb{Z}}

\newcommand{\R}{\mathbb{R}}

\newcommand{\1}{{\text{\Large $\mathfrak 1$}}}

\newcommand{\var}{\operatorname{var}}

\newcommand{\eqdist}{\stackrel{\text{(d)}}{=}}

\renewcommand{\phi}{\varphi}
\renewcommand{\epsilon}{\varepsilon}
\newcommand{\weaklimit}[1]{\xrightarrow[#1]{\text{\rm (d)}}}
\def\skelpath{ \substack{s\\ \mbox{ $\leadsto$ }}}
\def\mydot{\text{\tiny $\bullet$}}

\begin{document}

\title[Directed random graphs and the Tracy-Widom distribution]
{Convergence to the Tracy-Widom distribution
for longest paths in a directed random graph}

\author{Takis Konstantopoulos}
\author{Katja Trinajsti\'c}

\address{Department of Mathematics, Uppsala University\newline
P.O. Box 480\newline
751 06 Uppsala\newline
Sweden}

\email{takis@math.uu.se, katja@math.uu.se}
\urladdr{\url{www.math.uu.se/~takis}, \url{www.math.uu.se/~katja}}



\begin{abstract}
 We consider a directed graph on the 2-dimensional integer lattice,
placing a directed edge from vertex $(i_1,i_2)$ to $(j_1,j_2)$, whenever
$i_1 \le j_1$, $i_2 \le j_2$, with probability $p$, independently
for each such pair of vertices.
Let $L_{n,m}$ denote the maximum length of all paths contained in
an $n \times m$ rectangle. We show that there is a positive exponent $a$,
such that, if $m/n^a \to 1$, as $n \to \infty$, then a properly
centered/rescaled version of $L_{n,m}$ converges weakly to the Tracy-Widom
distribution. A generalization to graphs with non-constant probabilities
is also discussed. 
\end{abstract}

\maketitle

\section{Introduction}
Random directed graphs form a class of stochastic models with
applications in computer science (\citealp{ISONEW94}), biology 
(\citealp{COHNEW91,NEWM92,NEWCOH86}) and physics
(\citealp{IK12}). Perhaps the simplest of all such graphs is a directed version of 
the standard Erd\H{o}s-R\'enyi random graph (\citealp{BE84}) on $n$ vertices,
defined as follows: For each pair $\{i,j\}$ of distinct positive integers less than
$n$, toss a coin with probability of head equal to $p$, $0<p<1$, independently from
pair to pair; if head shows up then introduce an edge directed from
$\min(i,j)$ to $\max(i,j)$. There is a natural extension of this
graph to the whole of $\Z$ studied in detail in \cite{FK03}.
In particular, if we define the asymptotic growth rate $C=C(p)$,
as the a.s.\ limit of the maximum length of all paths between $1$ and $n$
divided by $n$, \cite{FK03} provide sharp bounds on $C(p)$ for
all values of $p \in (0,1)$.

A natural generalization arises when we replace the total order of the vertex
set by a partial order, usually implied by the structure of the vertex set.
In such a model, coins are tossed
only for pairs of vertices which are comparable in this partial order.
The canonical case is to consider, as a vertex set,
the 2-dimensional integer lattice $\Z \times \Z$, equipped 
with the standard component-wise partial order: $(i_1, i_2) \prec (j_1, j_2)$
if the two pairs are distinct and $i_1 \le i_2$, $j_1 \le j_2$.
Such a graph was considered in \cite{DFK12}.
In that paper, it was shown 
that if $L_{n,m}$ denotes the maximum length 
of all paths of the graph, restricted to $\{0,\ldots,n\} \times 
\{1, \ldots,m\}$, then there is a positive $\kappa$
(depending on $p$ and the fixed integer $m$),
such that
\begin{equation}
\label{Llimslab}
\bigg(\frac{L_{[nt],m}-Cnt}{\kappa \sqrt{n}}, ~ t \ge 0\bigg)
\xrightarrow[n \to \infty]{\text{(d)}} (Z_{t,m},~ t \ge 0),
\end{equation}
where $Z_{\mydot, m}$ is the stochastic process defined in terms of $m$ independent
standard Brownian motions, $B^{(1)}, \ldots, B^{(m)}$, via the formula
\[
Z_{t,m} := \sup_{0=t_0 < t_1 \cdots < t_{m-1} < t_m=t}
\sum_{j=1}^m [B^{(j)}_{t_j} - B^{(j)}_{t_{j-1}}], \quad t \ge 0.
\]
One can speak of $Z$ as a {\em Brownian directed percolation model}, the terminology
stemming from the picture of a ``weighted graph'' on 
$\R \times \{1, \ldots, m\}$ where the weight of a segment 
$[s,t] \times \{j\}$ equals the change $B^{(j)}_{t} - B^{(j)}_{s}$
of a Brownian motion. If a path from $(0,0)$ to $(t,m)$
is defined as a union 
$\bigcup_{j=1}^m [t_{j-1}, t_j] \times \{j\}$
of such segments, then $Z$ represents the maximum weight
of all such paths.

\cite{BAR2001}, answering an open question by \cite{GW91}, showed that 
\[
Z_{1,m} \eqdist \lambda_m,
\]
where $\lambda_m$ is the largest eigenvalue of a GUE matrix of 
dimension $m$. Since $Z_{\mydot,m}$ is $1/2$-self-similar,
we see that 
\[
Z_{t,m} \eqdist \sqrt{t} \lambda_m.
\]
Now, fluctuations of $\lambda_m$ around the centering 
sequence $2\sqrt{m}$ have been quantified by \cite{TW94}
who showed the existence of a limiting law, denoted by $F_{\text{TW}}$:
\[
m^{1/6} (\lambda_m-2\sqrt{m}) \xrightarrow[m \to \infty]{\text{(d)}} 
F_{\text{TW}}.
\]
A natural question then, raised in \cite{DFK12}, is whether
one can obtain $F_{\text{TW}}$ as a weak limit of $L_{n,m}$ when
$n$ and $m$ tend to infinity simultaneously.
Our paper is concerned with resolving this question.
To see what scaling we can expect, rewrite the last display, 
for arbitrary $t>0$,  as
\[
m^{1/6} (\frac{Z_{t,m}}{\sqrt{t}}-2\sqrt{m}) \xrightarrow[m \to \infty]{\text{(d)}} 
F_{\text{TW}}.
\]
A statement of the form $X(t,m) \xrightarrow[m \to \infty]{\text{(d)}}  X$,
where the distribution of $X(t,m)$ does not depend on the choice of $t>0$, 
implies the statement 
$X(t, m(t)) \xrightarrow[t \to \infty]{\text{(d)}}  X$, for any function $m(t)$
such that $m(t) \xrightarrow[t \to \infty]{} \infty$. Hence, 
upon setting $m=[t^a]$, we have
\begin{equation}
\label{Zdef}
t^{a/6} \bigg(\frac{Z_{t,[t^a]}}{\sqrt{t}}-2\sqrt{t^a}\bigg) 
\xrightarrow[t \to \infty]{\text{(d)}} 
F_{\text{TW}}.
\end{equation}
Therefore, it is reasonable to guess that, when $a$ is small enough,
an analogous limit theorem
holds for a centered scaled version of the largest 
length $L_{n,[n^a]}$, namely that
\begin{equation}
\label{Llim}
n^{a/6} \bigg(\frac{L_{n,[n^a]}-c_1 n}{c_2 \sqrt{n}}-2 \sqrt{n^a}\bigg) 
\weaklimit{n \to \infty}
F_{\text{TW}},
\end{equation}
where $c_1, c_2$  are appropriate constants.

A stochastic model, bearing some resemblance to ours, is the so-called
{\em directed last passage percolation model} on $\Z^d$ (the case $d=2$ being
of interest here). We are given a collection of i.i.d.\ random variables
indexed by elements of $\Z^d_+$. A path from the origin to the point 
$n \in \Z^d_+$ is a sequence of elements of $\Z^d_+$, starting from
the origin and ending at $n$, such that the difference of  
successive members of the sequence is equal to the unit vector
in the $i$th direction, for some $1 \le i \le d$.
The weight of a path is the sum of the random variables associated with
its members. Specializing to $d=2$, let $L_{n,m}$ be the largest weight of
all paths from $(0,0)$ to $(n,m)$. 
Assuming that the random variables have
a finite moment of order larger than $2$, \cite{BM2005}
showed that \eqref{Llim} holds for all sufficiently small positive $a$
(the threshold depending on the order of the finite moment).
Independently, \cite{BS2005}
obtained the same result for random variables with a finite $4$th moment
and for $a<3/14$. In both papers, partial sums of i.i.d.\ were 
approximated with Brownian motions,
in the first case using the Koml\'os-Major-Tusn\'ady (KMT) construction, 
while in the second using Skorokhod embedding.

To show that \eqref{Llim} holds for our model, we adopt the 
technique introduced in \cite{DFK12}, which involves 
the existence of {\em skeleton points}
on each line $\Z \times \{j\}$. Skeleton points are, by definition,
 random points which
are connected with all the other points
on the same line. In \cite{DFK12} Denisov, Foss and Konstantopoulos used this fact, 
together with the fact that,
for finite $m$, one can pick skeleton points common to all $m$ lines, in order
to prove \eqref{Llimslab}. However, when $m$ tends to infinity simultaneously
with $n$, it is not possible to pick skeleton points common to all lines.
Modifying the definition of skeleton points enables us to give a new proof 
of \eqref{Llimslab}, as well as to prove \eqref{Llim}.
To achieve the latter, we borrow the idea of KMT coupling from \cite{BM2005}.
However, we need to do some work in order to express the random variable
$L_{n,m}$ in a way that resembles a maximum of partial sums.

Although we focus on the case where the edge probability $p$
is constant, it is possible to consider a more general case, where the
probability that a vertex $(i_1,i_2) \in \Z\times \Z$ connects to a vertex $(j_1,
j_2)$
depends on the distances $|j_1-i_1|$ and $|j_2-i_2|$ of the two vertices.
This generalization is discussed in
the last section of the article.

\section{The one-dimensional directed random graph}
We summarize below some properties of the directed Erd\H{o}s-R\'enyi graph
on $\Z$ with connectivity probability $p$ taken from \cite{FK03}.
For $i<j$, let $L[i,j]$ be the maximum length of all paths with start
and end points in the interval $[i,j]$. Then, for $i < j < k$,
we have
$L[i,k] \le L[i,j]+ L[j,k] + 1$.
Since the distribution of the random graph is invariant under translations,
and is also ergodic (the natural invariant $\sigma$-field is trivial),
it follows, from Kingman's subadditive ergodic theorem, that
there is a deterministic constant $C=C(p)$ such that
\begin{equation}
\label{CCC}
\lim_{n \to \infty} L[1,n]/n = C, \text{ a.s.}
\end{equation}
 In fact,
$C = \inf_{n \ge 1} E L[1,n]/n$.
The function $C(p)$ is not known explicitly; only bounds are known 
(\citealp[Thm.\ 10.1]{FK03}). For
example, $0.5679 \le C(1/2) \le 0.5961$.
We also know that there exists, almost surely, a random integer sequence  
$\{\Gamma_r, r \in \Z\}$
with the property that for all $r$, all $i < \Gamma_r$, and all 
$j > \Gamma_r$,
there is a path from $i$ to $\Gamma_r$ and a path from $\Gamma_r$ to $j$.
The existence of such points, referred to as {\em skeleton points}, 
is not hard to establish 
(\citealp{DFK12}).
Since the directed Erd\H{o}s-R\'enyi graph is invariant
under translations, so is the sequence of skeleton points, i.e., 
$\{\Gamma_r, r \in \Z\}$ has the same law as $\{n+\Gamma_r, r \in \Z\}$, 
for all $n \in \Z$. Moreover, it turns out that the sequence forms a stationary
renewal process. If we enumerate the skeleton points according to
$\cdots < \Gamma_{-1} < \Gamma_0 \le 0 < \Gamma_1 < \cdots$, we
have that $\{\Gamma_{r+1}-\Gamma_r, r \in \Z\}$ are independent random variables,
whereas $\{\Gamma_{r+1}-\Gamma_r, r \not = 0\}$ are i.i.d.
Stationarity implies that the law of the omitted difference $\Gamma_1-\Gamma_0$
has a density which is proportional to the tail of the distribution of $\Gamma_2-
\Gamma_1$. In \cite{DFK12} it is shown that the distance
$\Gamma_{2}-\Gamma_1$ between two successive skeleton points
has a finite $2$nd moment. One can follow the same steps of the proof, 
to show that in our case, with constant probability $p$,
this random variable has moments of all orders.
Moreover, one can show that for some $\alpha>0$ 
(the maximal such $\alpha$ depends on $p$) it holds that
$Ee^{\alpha(\Gamma_{2}-\Gamma_1)}<\infty$. 

The rate $\lambda_0$ of the sequence of skeleton points can be expressed
as an infinite product:
\begin{equation}
\label{lambda}
\lambda_0 := \frac{1}{E(\Gamma_2-\Gamma_1)} = \prod_{k=1}^\infty
(1-(1-p)^k)^2.
\end{equation}
For example, for $p=1/2$, $\lambda_0 \approx 1/12$.

A central limit theorem for $L[1,n]$ is also available (\citealp [Thm. 2]{DFK12}).
If we let 
\begin{equation}
\label{sigma}
\sigma_0^2 := \var (L[\Gamma_1, \Gamma_2] - C(\Gamma_2-\Gamma_1)),
\end{equation}
then 
\begin{equation}
\label{L1}
\frac{L[1,n]-Cn}{\sqrt{\lambda_0 \sigma_0^2 n}} \weaklimit{n \to \infty} N(0,1),
\end{equation}
where $N(0,1)$ is a standard normal random variable.
Note that $\sigma_0^2 \not = \var (L[\Gamma_1, \Gamma_2])$.
Unfortunately, we have no estimates for $\sigma_0^2$, but, interestingly,
there is a technique for estimating it, based on perfect simulation.
This was briefly explained in \cite{FK03} in connection with an 
infinite-dimensional Markov chain which carries most of the information
about the law of the directed Erd\H{o}s-R\'enyi random graph.

In addition, it is shown in \cite{FK03} that $C$ can also be expressed as
\begin{equation}
\label{CCC2}
C = \frac{EL[\Gamma_1, \Gamma_2]}{E(\Gamma_2-\Gamma_1)}.
\end{equation}
In fact, if $\{\nu_r, r \in \Z\}$ is a random sequence of integers,
defined on the same probability space as the one supporting the
random graph, such that $\{\Gamma_{\nu_r}, r \in \Z\}$
is a stationary point process then
\[
C = \frac{EL[\Gamma_{\nu_{r}}, \Gamma_{\nu_{r+1}}]}{E(\Gamma_{\nu_{r+1}}-\Gamma_{\nu_r})}.
\]

The most important property of the skeleton points is that if $\gamma$
is a skeleton point, and if $i \le \gamma \le j$, then a path with
length $L[i,j]$ (a maximum length path) must necessarily contain $\gamma$.
This crucial property will be used several times below, especially since, 
for every $i<j$, the following equality holds
$$L[\Gamma_i,\Gamma_j]=L[\Gamma_i,\Gamma_{i+1}]+L[\Gamma_{i+1},
\Gamma_{i+2}]+\dots+L[\Gamma_{j-1},\Gamma_j].$$
Furthermore, the restriction of the graph on the interval between two successive 
skeleton points is independent of the restriction on the complement of the
interval; hence the summands in the right-hand side of the last display are
independent random variables.

\section{Statement of the main result}
It is clear from \eqref{L1} that the constants $c_1, c_2$ in \eqref{Llim}
should be as follows: $c_1=C$, $c_2=\sqrt{\lambda \sigma^2}$.
Now we can formulate the main result.
\begin{theorem}
\label{main}
Let $C$, $\lambda_0$, $\sigma_0^2$ be the quantities associated with
the directed random graph on $\Z$ with connectivity probability $p$,
defined by \eqref{CCC} (equivalently, \eqref{CCC2}), \eqref{lambda}, 
\eqref{sigma}, respectively.
Consider the directed random graph on $\Z \times \Z$ and let
$L_{n,m}$ be the maximum length of all paths between two vertices
in $[0,n] \times [1,m]$. Then, for all $0< a < 3/14$,
\begin{equation}
\label{Llim2}
n^{a/6} 
\bigg(\frac{L_{n,[n^a]}-C n}{\sqrt{\lambda_0 \sigma_0^2} \sqrt{n}}-2 \sqrt{n^a}\bigg) 
\weaklimit{n \to \infty}
F_{\text{TW}},
\end{equation}
where $F_{\text{TW}}$ is the Tracy-Widom distribution.
\end{theorem}
To prove this theorem, we will first define the notion of skeleton
points for the graph on $\Z \times \Z$
and then prove pathwise upper and lower bounds for $L_{n,m}$ which depend
on paths going through these skeleton points. 
This will be done in Section \ref{SP}.
In Section \ref{M1} we show that the difference between these bounds is
of the order $o(n^b)$, where $b=(1/2)-(a/6)$ is the net exponent in 
the denominator of \eqref{Llim2}.
We will then (Section \ref{M2}) 
introduce a quantity $S_{n,m}$ which resembles a last passage
percolation problem and show that it differs from $L_{n,m}$
by a quantity which is of the order $o(n^b)$, when $m=[n^a]$. 
The problem will then be translated to a last passage percolation
problem (with the exception of random indices). This will finally,
in Section \ref{M3} be compared to the Brownian directed percolation problem 
by means of strong coupling.

\section{Skeleton points and pathwise bounds}
\label{SP}
Our model is a directed random graph $G$ with vertices $\Z \times \Z$.
For each pair of vertices $\bm i$, $\bm j$, such that $\bm i \prec \bm j$,
toss an independent coin with probability of heads equal to $p$; if a head shows up
introduce an edge directed from $\bm i$ to $\bm j$. 

A path of length $\ell$ in the graph 
is a sequence $(\bm i_0, \bm i_1, \ldots, \bm i_\ell)$ of
vertices $\bm i_0\prec \bm i_1\prec\ldots\prec \bm i_\ell$ such that there is an 
edge between any consecutive vertices.

We denote by $G_{n,m}$ the restriction of $G$ on the set of
vertices $\{0,1,\ldots,n\} \times \{1,\ldots,m\}$. 
The random variable of interest is
\[
L_{n,m} := \text{ the maximum length of all paths in } G_{n,m}.
\]
We refer to the set $\Z \times \{j\}$ as ``line $j$'' or ``$j$th line'',
and note that the restriction of $G$ onto  $\Z \times \{j\}$
is a directed Erd\H{o}s-R\'enyi random graph. 
We denote this restriction by $G^{(j)}$. Typically, a superscript $(j)$ will
refer to a quantity associated with this restriction.
For example, for $a \le b$,
\begin{align*}
L^{(j)}[a,b] := &\text{ the maximum length of all paths in $G^{(j)}$}\\
&\qquad\quad\text{with vertices between $(a,j)$ and $(b,j)$}
\end{align*}
and we agree that $L^{(j)}[a,b] =0$ if 
$a\geq b$.

Clearly, the $\{G^{(j)}, j \in \Z\}$ are i.i.d.\ random graphs,
identical in distribution to the directed Erd\H{o}s-R\'enyi random graph.
Therefore, for each $j \in \Z$,
\[
\lim_{n \to \infty} L^{(j)}[1,n]/n = C, \text{ a.s.}
\]

To establish upper and lower bounds for $L_{n,m}$, we need to slightly change
the definition of a skeleton point in $G$.
\begin{definition}[Skeleton points in $G$]
\label{skel}
A vertex $(i,j)$ of the directed random graph $G$ is called skeleton
point if it is a skeleton point for $G^{(j)}$ (for any $i' < i < i''$,
there is a path from $(i',j)$ to $(i,j)$ and a path from $(i,j)$ to
$(i'', j)$) and if there is an edge from $(i,j)$ to $(i, j+1)$.
\end{definition}
Therefore, the skeleton points on line $j$ are obtained from the skeleton
point sequence of the directed Erd\H{o}s-R\'enyi random graph $G^{(j)}$ by
independent thinning with probability $p$.
When we refer to skeleton points on line $j$, we shall be speaking of this
thinned sequence. The elements of this sequence are denoted by
\[
\cdots <  \Gamma^{(j)}_{-1} < \Gamma^{(j)}_0 \le 0 < \Gamma^{(j)}_1 < \Gamma^{(j)}_2 
< \cdots
\]
and have rate
\[
\lambda = \frac{1}{E (\Gamma^{(j)}_2-\Gamma^{(j)}_1)}
= p \lambda_0 = p \prod_{k=1}^\infty (1-(1-p)^k)^2.
\]
The associated counting process of skeleton points on line $j$ is defined
by
\[
\Phi^{(j)}(t)-\Phi^{(j)}(s) = \sum_{r \in \Z} \1(s < \Gamma^{(j)}_r \le t),
\quad s, t \in \R, \quad s \le t,
\]
together with the agreement that 
\[
\Phi^{(j)}(0)=0.
\]
Note that we insist
on having the parameter $t$ in $\Phi^{(j)}(t)$ as an element of $\R$ (and not
just $\Z$).
We also let
\begin{align*}
X^{(j)}(t) &:= \Gamma^{(j)}_{\Phi^{(j)}(t)},
\\
Y^{(j)}(t) &:= \Gamma^{(j)}_{\Phi^{(j)}(t)+1},
\end{align*}
be the skeleton points on line $j$ straddling $t$:
\begin{equation}
\label{straddling}
X^{(j)}(t) \le t < Y^{(j)}(t).
\end{equation} 
Next we
prove upper and lower bounds for $L_{n,m}$. The set of dissections
of the interval $[0,n] \subset \R$ in $m$ non-overlapping, possibly empty intervals is denoted
by
\[
\mathcal T_{n,m} := 
\{\bm t =(t_0,t_1, \ldots, t_m) \in \R^{m+1}:~
0=t_0 \le t_1 \le \cdots \le t_{m-1} \le t_m=n\}.
\]

\begin{lemma}
\emph{(Upper bound)}
Define
\begin{equation}
\label{upperclaim}
\overline{L}_{n,m}:= \sup_{\bm t \in \mathcal T_{n,m}}
\sum_{j=1}^m L^{(j)}[X^{(j)}(t_{j-1}),~ Y^{(j)}(t_j)] + m.
\end{equation}
Then
$L_{n,m} \le \overline{L}_{n,m}$.
\end{lemma}
\begin{proof}
Let $\pi$ be a path in $G_{n,m}$. Consider the lines visited by $\pi$,
denoting their indices by
$1\le \nu_1 < \nu_2 < \cdots < \nu_J \le m$.
Let $(a_j, \nu_j)$ and $(b_j, \nu_j)$ be the first and the last vertex 
of line  $\nu_j$ in the path $\pi$.
Then the length of $\pi$ satisfies
\[
|\pi| \le \sum_{j=1}^J L^{(\nu_j)}[a_j, b_j] + J-1.
\]
Since successive vertices in the path should be increasing in the order $\prec$,
we have
$b_{j-1} \le a_j$, $2 \le j \le J$.
Hence, with $b_0:=0$, 
\[
|\pi| \le \sum_{j=1}^J L^{(\nu_j)}[b_{j-1}, b_j] + J-1
\le \sum_{j=1}^J L^{(\nu_j)}[X^{(\nu_j)}(b_{j-1}),~ Y^{(\nu_j)}(b_j)] + J-1,
\]
where we used \eqref{straddling}. 
Since $J \le m$, we can extend $0=b_0 \le b_1 \le \cdots \le b_J \le n$
to a dissection of $[0,n]$ into $m$ non-overlapping intervals, showing that
the right-hand side of the last display is bounded above by $\overline L_{n,m}$.
Taking the maximum over all $\pi$ in $G_{n,m}$, we obtain $L_{n,m}
\le \overline L_{n,m}$, as required.
\end{proof}

Note that the existence and properties of skeleton points were not
used in the proof of the upper bound, other than to ensure that
the upper bound is a.s.\ finite.

\begin{lemma}
\emph{(Lower bound)}
\label{lblemma}
Define
\begin{equation*}
\label{DELTA}
\Delta^{(j)}_n := \max_{0 \le i \le \Phi^{(j)}(n)} (\Gamma^{(j)}_{i+1}-
\Gamma^{(j)}_i),
\end{equation*}
and
\begin{equation*}
\underline{L}_{n,m}:=
\sup_{\bm t \in \mathcal T_{n,m}}
\sum_{j=1}^m L^{(j)}[Y^{(j)}(t_{j-1}),~ X^{(j)}(t_j)]
- \sum_{j=1}^m \Delta^{(j)}_n.
\end{equation*}
Then
$L_{n,m} \ge \underline{L}_{n,m}$.
\end{lemma}

\proof
We will show that, for all $\bm t=(t_0, \ldots, t_n) \in \mathcal T_{n,m}$,
there is a path $\pi$ in $G_{n,m}$ with length $|\pi|$ satisfying
\begin{equation}
\label{tpiestimate}
\sum_{j=1}^m L^{(j)}[Y^{(j)}(t_{j-1}),~ X^{(j)}(t_j)]
 \le |\pi| + \sum_{j=1}^m \Delta^{(j)}_n.
\end{equation}
Fix $\bm t\in  \mathcal T_{n,m}$ and use the notation 
\[
I_j =[Y^{(j)}(t_{j-1}),~ X^{(j)}(t_j)] = [a_j, b_j] ,
\quad j=1, \ldots,m.
\]
Note that $a_j\geq b_j$ if there is one or no skeleton points on the segment
$(t_{j-1}, t_j]\times \{j\}$ and then $L^{(j)}(I_j)=0$.

Given two skeleton points $(x,i)$, $(y,j)$ we
say that there is a {\em staircase path} from $(x,i)$ to $(y,j)$ if
there is a sequence of skeleton points
\[
(x,i)=(x_0,i),~ (x_1,i+1), \ldots,~ (x_{j-i}, j) = (y,j),
\]
such that $x=x_0  \le x_1 \le \cdots \le x_{j-i}=y$.
See Figure \ref{staircasefig}.
Clearly then, there is  a path from $(x,i)$ to $(y,j)$ which jumps upwards by
one step each time it meets a new skeleton point from the sequence.
We denote this by
\[
(x,i) \skelpath (y,j).
\]

\begin{figure}[ht]
 \centering
  \includegraphics[width=9cm]{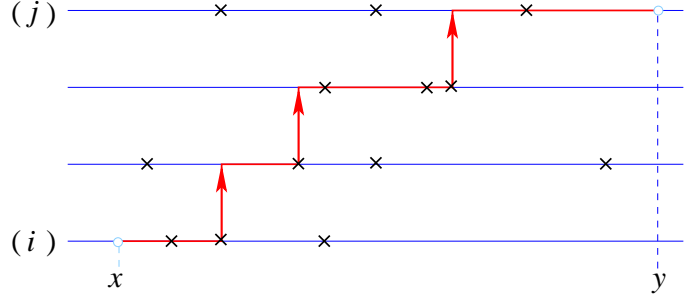}
  \caption{A staircase path from $(x,i)$ to $(y,j)$ jumps upwards
at skeleton points (denoted by $\sf x$) but may skip several of them
before deciding to make a jump}\label{staircasefig}
\end{figure}

Among all the staircase paths from $(x,i)$ to $(y,j)$, we will consider the
{\em best one}, defined by two properties:
\begin{itemize}
\item
Property 1: A best path from $(x,i)$ to $(y,j)$ jumps from line $k$ 
to line $k+1$, $k=i,i+1,\dots,j-1$, at the first next skeleton point on line $k$,
i.e.\ at the points $x_0$ and 
$x_{k-i+1}=Y^{(k+1)}(x_{k-i}),\ k=i,\dots,j-2.$
\item
Property 2: Every horizontal segment of a best path is a path of
maximal length.
\end{itemize}

If all the intervals $I_1, \ldots, I_m$
are empty, the left-hand side of \eqref{tpiestimate} is zero and
the inequality is trivially satisfied for any path $\pi$. 

Otherwise, for a fixed $\bm t \in G_{n,m}$ 
we will construct a path $\pi$ in $G_{n,m}$ for which \eqref{tpiestimate} holds.
Define 
a subsequence $\nu_1 < \nu_2 < \cdots$  of $1, \ldots, m$, inductively, as follows:
\begin{align}
\nu_1&:= \inf\{1\le j \le m:~ I_j \not=\varnothing\},
\label{proc1}
\\
\nu_r&:= \inf\{j > \nu_{r-1} :~ (b_{\nu_{r-1}}, \nu_{r-1}) \skelpath (b_j,j)\},
\quad r \ge 2. 
\label{proc2}
\end{align}
See Figure \ref{beststaircasefig} for an illustration.
The procedure stops if one of the elements of the subsequence
exceeds $m$ or if the condition inside the infimum is not satisfied by
a path in $G_{n,m}$. 

\begin{figure}[ht]
 \centering
  \includegraphics[width=12cm]{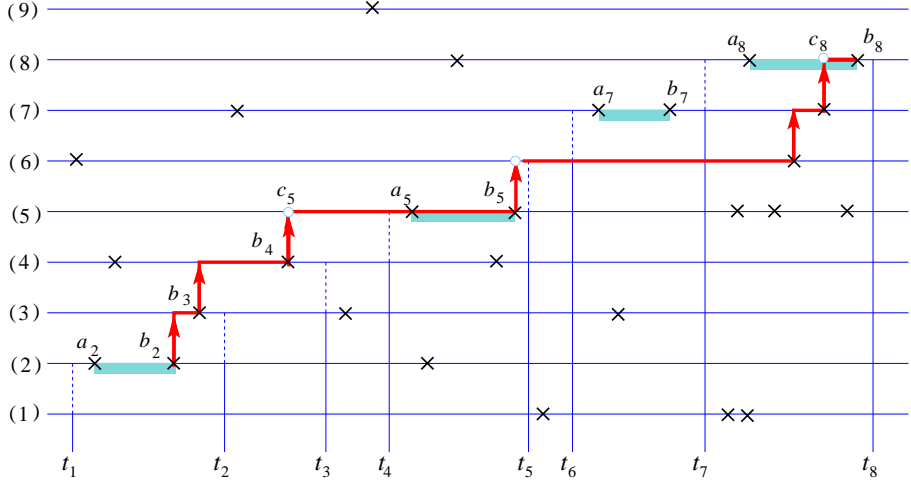}
  \caption{Illustration of  the procedure defined by \eqref{proc1}-\eqref{proc2}.
There are four best staircase paths:
the path from $(b_2, 2)$ to $(b_3,3)$,
the path from $(b_3, 3)$ to $(b_4,4)$,
the path from $(b_4, 4)$ to $(b_5,5)$,
and the path from $(b_5, 5)$ to $(b_8,8)$.
Observe that $I_j = (a_j, b_j)$, in the figure, are
nonempty only for $j=2,5, 7$ and $8$ (these are the highlighted intervals),
but $I_7$ is not visited by the constructed path.
Moreover, $I_8$ is only partly visited and the 
path enters $I_8$ at a point $c_8$ between $b_8$.}
\label{beststaircasefig}
\end{figure}

Let $J$ be the 
last index in the above defined sequence. 
Let $\pi_1$ be a path of maximum
length from $(a_{\nu_1}, \nu_1)$ to $(b_{\nu_1}, \nu_1)$ and 
define, for $r=2,3,\dots,J$,
a path $\pi_r$ as a best staircase path 
from $(b_{\nu_{r-1}}, \nu_{r-1})$ to $(b_{\nu_r},\nu_r)$.
Note that, for each $r=2, 3, \ldots, J$, $J \geq 2$, the end vertex
of $\pi_{r-1}$ is the
start vertex of $\pi_r$. Therefore we can concatenate the paths
$\pi_1, \ldots, \pi_{J}$ to obtain a path $\pi$.
This path starts from $(a_{\nu_1}, \nu_1)$ and ends at $(b_{\nu_J}, \nu_J)$.

Let
\[
\pi^{(j)} :=
\text{ the restriction of path $\pi$ on line $j$}
\]
and $|\pi^{(j)}|$ its length on line $j$.
Also, for $j\geq \nu_1$ denote 
\[
\text{$(c_j, j):=$ the first vertex on line $j$ of path $\pi$.}
\]

Split the sum in the left-hand side of \eqref{tpiestimate} along
the elements of the subsequence $\{\nu_1, \ldots, \nu_J\}$:
\begin{gather*}
\sum_{j=1}^m L^{(j)}(I_j) = \sum_{r=1}^{J+1} 
\sum_{j=\nu_{r-1}+1}^{ \nu_r} L^{(j)}(I_j) =: \sum_{r=1}^{J+1} G_r,
\end{gather*}
where we have conveniently set 
\[
\text{$\nu_0:=0$, $\nu_{J+1}:=m$,}
\]
in order to take care of the first and last terms. 
By the defintion of $\nu_1$, the intervals $I_1,I_2,\dots,I_{\nu_{1}-1}$ are 
empty and 
$$G_1=L^{(\nu_1)}(I_{\nu_1})=\vert\pi^{(\nu_1)}\vert.$$
Assume now that $2\le r \le J$, and 
write
\[
G_r = \sum_{j=\nu_{r-1}+1}^{\nu_r} L^{(j)}(I_j) =
 \sum_{j=\nu_{r-1}+1}^{\nu_r-1} L^{(j)}(I_j) + L^{({\nu_r})}(I_{\nu_r}).
\]
Since $\pi_{\nu_r}$ 
is the path of maximal length from $(c_{\nu_r},\nu_r)$ to its end-vertex 
$(b_{\nu_r},\nu_r)$ (Property 2), if $c_{\nu_r}<a_{\nu_r}$ 
then $L^{({\nu_r})}(I_{\nu_r})\le  |\pi^{(\nu_r)}|$.
Define in this case $I_{\nu_r}'=\emptyset$.
Otherwise, we can write 
\[
L^{({\nu_r})}(I_{\nu_r}) \leq L^{({\nu_r})}[a_{\nu_r}, Y^{(\nu_r)}(c_{\nu_r})]
+ L^{({\nu_r})}[Y^{(\nu_r)}(c_{\nu_r}), b_{\nu_r}].
\]
Then, again because of Property 2, 
$L^{({\nu_r})}[Y^{(\nu_r)}(c_{\nu_r}), b_{\nu_r}]<|\pi^{(\nu_r)}|$
and it is left to find a bound on the interval $I_{\nu_r}'=[a_{\nu_r},Y^{(\nu_r)}(c_{\nu_r})]$.
Recall that by Property 1, depending whether $j$ 
is a member of the sequence $\{\nu_r,r=1,\dots,J\}$ or not,
 $c_{j+1}=b_j$ or $c_{j+1}=Y^{(j)}(c_{j})$, respectively. 
Also, because of $I_j\subseteq [t_{j-1},t_j]$, we know that $L^{(j)}(I_j)\leq t_j-t_{j-1}$.
Hence, if $\nu_{r}-\nu_{r-1}>1$ it holds
\begin{align*}
 \sum_{j=\nu_{r-1}+1}^{\nu_r-1} L^{(j)}(I_j)+L^{(j)}(I_{\nu_r}')\leq 
   \sum_{j=\nu_{r-1}+1}^{\nu_r-1} (t_j-t_{j-1}) + (Y^{(\nu_r)}(c_{\nu_r})-t_{\nu_r-1})&\\
\leq   Y^{(\nu_r)}(c_{\nu_r})-b_{\nu_{r-1}}
=  \sum_{j=\nu_{r-1}+1}^{\nu_r} (Y^{(j)}(c_{j})-c_j)\leq
    \sum_{j=\nu_{r-1}+1}^{\nu_r}  \Delta^{(j)}_n.&
\end{align*}
Combining the above, we obtain
\begin{align*}
G_r &\le \sum_{j=\nu_{r}+1}^{\nu_r-1} \Delta^{(j)}_n + |\pi^{(\nu_r)}|.
\end{align*}
If $\nu_J=m$, then $G_{J+1}=0$. Otherwise, we can extend the sequence 
$\{c_j,j=\nu_1,\nu_1+1,\dots,\nu_J\}$ defining iteratively
$c_{\nu_J+1}:=b_{\nu_J}$ and $c_{j+1}:=Y^{(j)}(c_j)$
until $c_j>n$ for some $j$. Let $K$ be the last index such that $c_K\leq n$.
As there was not possible to construct the best staircase path after the 
line $\nu_J$, $K$ is at most $m$.
Similarly as above, for $G_{J+1}$ it holds
\begin{align*}
 G_{J+1}&=\sum_{j=\nu_{J}+1}^m L^{(j)}(I_j)\leq 
   \sum_{j=\nu_{J}+1}^m (t_j-t_{j-1}) 
\leq   n-b_{\nu_J}\\
&= \sum_{j=\nu_{J}+1}^K (Y^{(j)}(c_j)-c_j)\leq
    \sum_{j=\nu_{J}+1}^K \Delta^{(j)}_n.
\end{align*}

Finally, we obtain
\[
\sum_{j=1}^m L^{(j)}(I_j) \le \sum_{j=1}^m \Delta^{(j)}_n + \sum_{r=1}^J |\pi^{(\nu_r)}|
\le \sum_{j=1}^m \Delta^{(j)}_n + |\pi|,
\]
as required.
\qed

\section{Further estimates in probability and Brownian directed percolation}
\label{M}

In the present section we prove Theorem \ref{main} as a sequence of lemmas.

\subsection{Asymptotic coincidence of the two bounds}
\label{M1}

Looking at \eqref{Llim}, 
we can see that the correct scaling requires exponent
\[
b:= \frac{1}{2}-\frac{a}{6}
\]
in the denominator and condition $a< 3/7$, which is equivalent to $a<b$.

In the following two lemmas we will not specifically use the 
definiton of $b$ and condition on $a$. Both lemmas hold for more general 
$a,b>0$, $0<b-a<1$.

\begin{lemma}
\label{ac}
With $b=(1/2)-(a/6)$ and $a < 3/7$,
\[
\frac{\overline L_{n,[n^a]}- \underline L_{n,[n^a]}}{n^b} 
\xrightarrow[n \to \infty]{\text{\emph{(p)}}}   0.
\]

\end{lemma}

\proof
Let $\bm t$ be such
that the maximum in the right-hand side of \eqref{upperclaim} is achieved.
Then
\begin{align*}
\overline L_{n,m} - \underline L_{n,m}
&\le m + \sum_{j=1}^m \Delta^{(j)}_n \\
& \qquad + \sum_{j=1}^m L^{(j)}[X^{(j)}(t_{j-1}), Y^{(j)}(t_j)]
-\sum_{j=1}^m L^{(j)}[Y^{(j)}(t_{j-1}), X^{(j)}(t_j)]
\\
&\leq m + \sum_{j=1}^m \Delta^{(j)}_n\\
& \qquad + \sum_{j=1}^m
\bigg\{
L^{(j)}[X^{(j)}(t_{j-1}),  Y^{(j)}(t_{j-1})]
+ L^{(j)}[X^{(j)}(t_j), Y^{(j)}(t_j)]
\bigg\}
\\
&\le m + \sum_{j=1}^m \Delta^{(j)}_n +
2 \sum_{j=1}^m \max_{0 \le i \le \Phi^{(j)}(n)}  L^{(j)}[\Gamma^{(j)}_i, 
\Gamma^{(j)}_{i+1}]
\\
& \le m + 3 \sum_{j=1}^m \Delta^{(j)}_n.
\end{align*}
Hence
\[
\frac{1}{n^b}E[\overline L_{n,[n^a]} - \underline L_{n,[n^a]}]
\le  \frac{n^a}{n^b} + 3 \frac{n^a}{n^b} E[\Delta_n^{(1)}].
\]
Since $b>a$ and the random variables $\{\Gamma^{(1)}_{i+1} - \Gamma^{(1)}_{i},i\geq 1\}$ have a finite
$1/(b-a)$-th moment, the converegence to 0 for the second term above follows by Lemma \ref{max1}.
\qed

\subsection{Centering}
\label{M2}
We introduce the quantity
\[
S_{n,m} := \sup_{\bm t \in \mathcal T_{n,m}} 
\sum_{j=1}^m \bigg\{ L^{(j)}[X^{(j)}(t_{j-1}), X^{(j)}(t_j)]
-C [X^{(j)}(t_j) - X^{(j)}(t_{j-1})]\bigg\}.
\]
This should be ``comparable'' to $L_{m,n}-Cn$ when $m=[n^a]$. Indeed, we have:

\begin{lemma}
\label{clemma}
With $b=(1/2)-(a/6)$, and $a < 3/7$, 
\[
\frac{S_{n,[n^a]} -(L_{n,[n^a]}-Cn)}{n^b} \xrightarrow[n \to \infty]{\text{\emph{(p)}}}  0.
\]

\end{lemma}
\proof
We begin by rewriting the numerator above as
           \begin{multline*}
S_{n,m} -(L_{n,m}-Cn) 
\\ 
= \sup_{\bm t \in \mathcal T_{n,m}} 
\bigg\{\sum_{j=1}^m L^{(j)}[X^{(j)}(t_{j-1}), X^{(j)}(t_j)]
+ Cn - C \sum_{j=1}^m  [X^{(j)}(t_j) - X^{(j)}(t_{j-1})] \bigg\}
-L_{n,m}.
\end{multline*}
Upon writing $n = \sum_{j=1}^m (t_j-t_{j-1})$, for any $\bm t \in \mathcal T_{n,m}$,
we have
\[
\bigg|n-\sum_{j=1}^m  [X^{(j)}(t_j) - X^{(j)}(t_{j-1})]\bigg|
=\bigg| \sum_{j=1}^m [t_j - X^{(j)}(t_j)] - \sum_{j=1}^m 
 [t_{j-1}- X^{(j)}(t_{j-1})] \bigg|
\le 2 \sum_{j=1}^m \Delta^{(j)}_n.
\]
Hence, on the one hand we have
\begin{align*}
S_{n,m} -(L_{n,m}-Cn)
&\le \sup_{\bm t \in \mathcal T_{n,m}}
\sum_{j=1}^m L^{(j)}[X^{(j)}(t_{j-1}), X^{(j)}(t_j)]
-L_{n,m} + 2 C\sum_{j=1}^m \Delta^{(j)}_n
\\
&\le \sup_{\bm t \in \mathcal T_{n,m}}
\sum_{j=1}^m L^{(j)}[X^{(j)}(t_{j-1}), Y^{(j)}(t_j)]
-L_{n,m} + 2C\sum_{j=1}^m \Delta^{(j)}_n
\\
&\le\overline L_{n,m} -\underline L_{n,m} + 2 C\sum_{j=1}^m \Delta^{(j)}_n.
\\
\end{align*}
On the other hand,
\begin{align*}
S_{n,m} -(L_{n,m}-Cn)
&\ge \sup_{\bm t \in \mathcal T_{n,m}}
\sum_{j=1}^m L^{(j)}[X^{(j)}(t_{j-1}), X^{(j)}(t_j)]
-L_{n,m} - 2C \sum_{j=1}^m \Delta^{(j)}_n
\\
&\ge \sup_{\bm t \in \mathcal T_{n,m}}
\sum_{j=1}^m L^{(j)}[Y^{(j)}(t_{j-1}), X^{(j)}(t_j)]
-L_{n,m} - 2 C\sum_{j=1}^m \Delta^{(j)}_n
\\
&\ge \underline L_{n,m} -\overline L_{n,m} - 2C \sum_{j=1}^m \Delta^{(j)}_n
 = - (\overline L_{n,m} - \underline L_{n,m}) -2C\sum_{j=1}^m \Delta^{(j)}_n.
\end{align*}
Therefore, 
\[
|S_{n,m} - (L_{n,m} - Cn)| \le 
\overline L_{n,m} - \underline L_{n,m} + 2C \sum_{j=1}^m \Delta^{(j)}_n.
\]
Thus, for $m=[n^a]$,  the result follows by applying 
Lemma \ref{max1} and Lemma \ref{ac}.
\qed

Define now variance $\sigma^2$ as
$$\sigma^2:=\var (L^{(j)}[\Gamma^{(j)}_{k-1}, \Gamma^{(j)}_k]
-C( \Gamma^{(j)}_k -\Gamma^{(j)}_{k-1}))$$
and observe that $\sigma^2=\sigma_0^2/p$.
We work with the quantity
$\frac{1}{\sigma} S_{n,m}$, which can be rewritten as
\[
\frac{1}{\sigma} S_{n,m} = \sup_{\bm t \in \mathcal T_{n,m}}
\sum_{j=1}^m \sum_{k= \Phi^{(j)}(t_{j-1})+1}^{\Phi^{(j)}(t_j)}
\chi^{(j)}_k,
\]
where
\[
\chi^{(j)}_k:= 
\frac{1}{\sigma}\big\{
L^{(j)}[\Gamma^{(j)}_{k-1}, \Gamma^{(j)}_k]
-C( \Gamma^{(j)}_k -\Gamma^{(j)}_{k-1})
\big\}.
\]
Note that the random variables $\{\chi_k^{(j)}\}_{k \ge 1, j\ge 1}$, 
indexed by
both $k$ and $j$, are independent and that $\{\chi_k^{(j)}\}_{k \ge 2, j \ge 1}$
are identically distributed with zero mean and unit variance.
The fact that the $\{\chi_1^{(j)}\}_{j\ge 1}$ do not have the same distribution
will not affect the result, so we will not separately take care of it.

\subsection{Coupling with Brownian motion}
\label{M3}

The term $\frac{1}{\sigma} S_{n,m}$ resembles a
centered last passage percolation
path weight, except that random indices are involved.
Therefore, we start using the idea of strong coupling
with Brownian motions, analogously to the proof in \cite{BM2005}.
Let $B^{(1)}, B^{(2)}, \ldots$ be i.i.d.\ standard Brownian
motions, and recall that
\begin{equation}
\label{Zexp}
Z_{n,m} := \sup_{\bm t \in \mathcal T_{n,m}}
\sum_{j=1}^m [B^{(j)}_{t_j} - B^{(j)}_{t_{j-1}}].
\end{equation}
Define the random walks $R^{(1)}, R^{(2)}, \ldots$ by
\begin{equation}
\label{RRR}
R^{(j)}_i = \sum_{k=1}^i \chi^{(j)}_k, \quad i = 0,1,2, \ldots,
\end{equation}
with $R^{(j)}_0=0$. With this notation, we have
\begin{equation}
\label{Sexp}
\frac{1}{\sigma} S_{n,m} = \sup_{\bm t \in \mathcal T_{n,m}}
\sum_{j=1}^m [R^{(j)}_{\Phi^{(j)}(t_j)} - R^{(j)}_{\Phi^{(j)}(t_{j-1})}].
\end{equation}
Taking into account \eqref{Zdef} and Lemma \ref{clemma}, 
it is evident that to prove Theorem \ref{main}
it remains to show that
\[
\frac{\sigma^{-1} S_{n,[n^a]} - \sqrt{\lambda} Z_{n,[n^a]}}{n^b}
\xrightarrow[n \to \infty]{\text{(p)}}0,
\]
or, using the scaling property of Brownian motion, that:
\begin{lemma}
\label{limlaw}
For all $a<3/14$, 
\begin{equation*}
\frac{\sigma^{-1} S_{n,[n^a]} - Z_{\lambda n,[n^a]}}{n^b}
\xrightarrow[n \to \infty]{\text{\emph{(p)}}}0.
\end{equation*}
\end{lemma}

To show that the random walks are close enough to the Brownian motion we use
the following version of the Koml\'os-Major-Tusn\'ady strong approximation 
result (\citealp [Thm.\ 4]{KMT76}):
\begin{theorem}
\label{KMTcoupling}
For any $0<r<1$, $n \in \N$ and $x\in \left[c_1(\log n)^{1/r},c_2(n\log n)^{1/2}\right]$, 
starting with a probability space supporting independent
Brownian motions $B^{(j)}$, $j=1,2,\ldots$, we can jointly construct
i.i.d.\ sequences $\chi^{(j)} = (\chi^{(j)}_1, \chi^{(j)}_2,\ldots)$,
$j=1,2,\ldots$, with the correct distributions and,
moreover, such that, with $R^{(j)}_i$ as in \eqref{RRR} above,
\begin{equation*}
P(\max_{1\le i \le n} |B^{(j)}_i - R^{(j)}_i| > x) \le C n \exp\{-\alpha x^r\} 
\text{\ for all\ } j=1,2,\ldots,
\end{equation*}
where the constants $C, c_1,c_2$ are depending on $\alpha$, $r$ and 
distributions of $\chi_1^{(1)}$ and $\chi_2^{(1)}$.
\end{theorem}

In addition to this, in order to take care of 
the random indices appearing in \eqref{Sexp},
we need a convergence rate result for 
the counting processes $\{\Phi^{(j)},j\geq 1\}$
which is proven in the appendix.



\proof[Proof of Lemma \ref{limlaw}]
From \eqref{Sexp} and \eqref{Zexp} we have
\begin{align*}
|\sigma^{-1} S_{n,m} - Z_{\lambda n, m} |
&\le
\sup_{\bm t \in \mathcal T_{n,m}}
\sum_{j=1}^m \bigg\{
\big\lvert ( R^{(j)}_{\Phi^{(j)}(t_j)} - R^{(j)}_{\Phi^{(j)}(t_{j-1})})
- (B^{(j)}_{\lambda t_j} - B^{(j)}_{\lambda t_{j-1}})\big\rvert 
\bigg\}
\\
&= 
\sup_{\bm t \in \mathcal T_{n,m}}
\sum_{j=1}^m \bigg\{
\big\lvert R^{(j)}_{\Phi^{(j)}(t_j)} - B^{(j)}_{\Phi^{(j)}(t_{j})}\big\rvert
+
\big\lvert R^{(j)}_{\Phi^{(j)}(t_{j-1})} - B^{(j)}_{\Phi^{(j)}(t_{j-1})}\big\rvert
\\
&\qquad \qquad\qquad\qquad +
\big\lvert B^{(j)}_{\Phi^{(j)}(t_j)} - B^{(j)}_{\lambda t_j}\big\rvert
+
\big\lvert B^{(j)}_{\Phi^{(j)}(t_{j-1})} - B^{(j)}_{\lambda t_{j-1}}\big\rvert
\bigg\}
\\
&\le
2 \sum_{j=1}^m \bigg\{
\max_{0 \le i \le n} \big\lvert R^{(j)}_i - B^{(j)}_i\big\rvert
+ \sup_{0 \le s \le n} \big\lvert B^{(j)}_{\Phi^{(j)}(s)} - B^{(j)}_{\lambda s}\big\rvert
\bigg\}
\\
&=: 2 \sum_{j=1}^m U^{(j)}_n + 2 \sum_{j=1}^m V^{(j)}_n,
\end{align*}
where
\[
U^{(j)}_n:= \max_{0 \le i \le n} \big\lvert R^{(j)}_i - B^{(j)}_i\big\rvert,
\qquad
V^{(j)}_n := \sup_{0 \le s \le n} \big\lvert B^{(j)}_{\Phi^{(j)}(s)} - B^{(j)}_{\lambda s}\big\rvert.
\]
Therefore, it is enough to show that,
\[
\frac{1}{n^b} \sum_{j=1}^{[n^a]} U^{(j)}_n \xrightarrow[n \to \infty]{\text{(p)}}0
\quad\text{and}\quad
\frac{1}{n^b} \sum_{j=1}^{[n^a]} V^{(j)}_n \xrightarrow[n \to \infty]{\text{(p)}}0.
\]
For the first convergence we will take into account the coupling estimate
as in Theorem \ref{KMTcoupling}.
That is, we will throughout assume that the random walks and
Brownian motions have been constructed {\em jointly}.
The second convergence will be established without this estimate, i.e.,
we will show that it is true, {\em regardless} of the joint construction
of the Brownian motions and the random walks.
This is because {\em the coupling we use is not detailed enough to give
us information about the joint distribution of $B^{(j)}$ and
the counting process $\Phi^{(j)}$} (the latter is not
a function of the random walks used in the coupling).

\emph{Proof of the first convergence.}
Let $\delta>0$. We need to show that
\[
P\big(\sum_{j=1}^{[n^a]} U^{(j)}_n > \delta n^b\big) \to 0 \quad \text{ as $n
\to \infty$.}
\]
Let $\epsilon<b-a$.
Then
\begin{equation}
\label{1st}
P\big(\sum_{j=1}^{[n^a]} U^{(j)}_n > \delta n^b\big)
\le  P\big( \max_{1\le j \le [n^a]} U^{(j)}_n  \le n^\epsilon, 
\sum_{j=1}^{[n^a]} U^{(j)}_n > \delta n^b\big)
+P\big(\max_{1\le j \le [n^a]} U^{(j)}_n > n^\epsilon\big).
\end{equation}
The first term from right-hand side is zero for large $n$. 
We are allowed, for $n$ large enough, to estimate the second term
using Theorem \ref{KMTcoupling} for an arbitrary $r\in(0,1)$ and $x=n^\epsilon$ as
\[
P\big(\max_{1\le j \le [n^a]} U^{(j)}_n > n^\epsilon\big)
\le n^a P\big(\max_{1\le i \le n} |B^{(1)}_i - R^{(1)}_i| > n^\epsilon\big)
\le n^a C n \exp\{-\alpha n^{\epsilon r}\} \to 0,
\]
as $n \to \infty$.

\emph{Proof of the second convergence.}
Let $1/4<\epsilon<b-a$. 
Replacing $U_n^{(j)}$ by $V_n^{(j)}$ in \eqref{1st},
we see that the first term is again zero for large $n$ and
it remains to show that second term converge to 0, i.e.,
it is enough to show the convergence for its upper-bound 
$$n^aP\big(V_n^{(1)} > n^\epsilon\big)\rightarrow 0,$$
as $n \rightarrow \infty$.
Let $\gamma>1$ and $1/2< q<2\epsilon$. Then we can write
\begin{align*}
n^aP\big(V^{(1)}_n > n^\epsilon\big)&\leq
     n^aP\big(\sup_{0\leq s\leq 2n} \lvert \Phi^{(1)}(s)-\lambda s\rvert> \gamma n^q\big) \\
&+ n^aP\big(\sup_{0\leq s \leq n}
    |B^{(1)}_{\Phi^{(1)}(s)} - B^{(1)}_{\lambda s}| > n^\epsilon,
    \sup_{0\leq s\leq 2n} \lvert \Phi^{(1)}(s)-\lambda s\rvert\leq \gamma n^q\big).
\end{align*}
Corollary \ref{convrate2} implies the convergence of the first term above
\begin{align*}
 n^aP\big(\sup_{0\leq s\leq 2n} \lvert \Phi^{(1)}(s)-\lambda s\rvert> \gamma n^q\big)
\leq n^a\exp\{-\alpha (2n)^{qr}\}\rightarrow 0,
\end{align*}
as $n\rightarrow \infty$, where $0<r<2-1/q$.
Set $\phi=\gamma n^q/\lambda$. For the second term using the fact that, under our condition,
$\lambda s-\gamma n^q\leq\Phi^{(1)}(s)\leq\lambda s+ \gamma n^q$ for $s\in [0,2n]$,
we have
\begin{align}
&n^aP\big(\sup_{0\leq s \leq n}         \nonumber
    |B^{(1)}_{\Phi^{(1)}(s)} - B^{(1)}_{\lambda s}| > n^\epsilon,
    \sup_{0\leq s\leq 2n} \lvert \Phi^{(1)}(s)-\lambda s\rvert\leq \gamma n^q\big)\\ \nonumber
&\leq n^aP\big(\max_{0\leq k\leq \lfloor n/\phi\rfloor}\sup_{0\leq s \leq \phi}
   |B^{(1)}_{\Phi^{(1)}(k\phi+s)} - B^{(1)}_{\lambda (k\phi+s)}| > n^\epsilon,\\ \nonumber
&\qquad \qquad\qquad\qquad\qquad\qquad \qquad 
  \sup_{0\leq s\leq 2n} \lvert \Phi^{(1)}(s)-\lambda s\rvert\leq \gamma n^q\big)\\ \nonumber
&\leq n^aP\big(\max_{0\leq k\leq\lfloor n/\phi\rfloor}
   \sup_{\substack{0\leq s \leq \phi\\ -\gamma n^q\leq t \leq \gamma n^q}}
   |B^{(1)}_{\lambda(k\phi+s)+t} - B^{(1)}_{\lambda (k\phi+s)}| > n^\epsilon\big)\\ \nonumber
&\leq n^a(n/\phi+1)P\big(\sup_{0\leq s\leq t\leq 3\gamma n^q}
    |B_{t} - B_{s}|>n^{\epsilon}\big)\\\nonumber
&\leq n^a(n/\phi+1)P\big(\sup_{0\leq s\leq 3\gamma n^q}
         \lvert B_{s}\rvert>n^{\epsilon}/2\big)\\ 
&\leq 4n^a(n/\phi+1)P(B_{3\gamma n^q}>n^{\epsilon}/2) \tag{a} \\
&\leq 4n^a(n/\phi+1)\exp\{-n^{2\epsilon}/(24\gamma n^q)\}\rightarrow 0, \tag{b}
\end{align}
as $n\rightarrow \infty$. 
The inequality (a) is the consequence of 
$P(\sup_{0<s<t} B_s>x)=2P(B_t>x)$ and the
inequality (b) is an estimate for the tail of the normal distribution. 
\qed

\begin{remark}
\label{rem3/14}
The condition $a<3/14$ is equivalent to $b-a>1/4$  and, thus, it
ensures existence of an $\epsilon>1/4$ and later existance of
$1/2<q<2\epsilon$. 
Therefore, it is necessary for application of Lemma \ref{convrate2},
 as well as for the convergence of the last estimate above. 
\end{remark}

\section{Graph with non-constant edge probabilities}

In \cite{DFK12}, the authors consider a one-dimensional model with connectivity probabilities 
depending on the distance between points, i.e.\ there is an edge between the vertices $i$ and $j$
with probability $p_{\lvert i-j\rvert}$. 
To prove a central limit theorem for the maximal path length in that graph, two conditions are introduced:
\begin{itemize}
 \item $0<p_1<1$;
 \item $\sum_{k=1}^{\infty}k(1-p_1)\cdots(1-p_k)<\infty$.
\end{itemize}
The conditions guarantee the existence of skeleton points and finite variance, defined as 
in \eqref{sigma}.

We can extend the idea of the graph on $\Z\times \Z$ keeping, whenever possible,
the notation from the $\Z\times \Z$ graph with constant probability $p$. 

Let $\{p_{i,j},i,j\leq 0 \}$ be a sequence of probabilities
that satisfies the following:
\begin{itemize}
 \item $0<p_{1,0}<1$; 
 \item $0<p_{0,1}<1$;
 \item $\sum_{k=1}^{\infty} k^{r-1}(1-p_{1,0})(1-p_{2,0})\cdots(1-p_{k,0})<\infty$ for some $r>2$.
\end{itemize}
Then the vertices $(i_1,i_2)$ and $(j_1,j_2)$, $(i_1,i_2)\prec (j_1,j_2)$, are connected with 
probability $p_{j_1-i_1,j_2-i_2}$.

The conditions above are needed to ensure the existence of skeleton points, 
as in Definition \ref{skel},
and a finite $r$th moment of the random variables $\chi_2,\chi_3,\ldots$ for some $r>2$.

From now on, let $r$ denote the order of the highest finite moment, i.e.\
$$r=\sup\{q>0:\sum_{k=1}^{\infty}k^{q-1}(1-p_{1,0})(1-p_{2,0})\cdots(1-p_{k,0})<\infty\}.$$

Now we can state the analogue of Theorem \ref{main} for this more general setting:
\begin{theorem}
 For all $a < \min(3/14,(r-2)/(3r/7+1))$,
\begin{equation*}
n^{a/6} 
\bigg(\frac{L_{n,[n^a]}-C n}{\sqrt{\lambda \sigma^2} \sqrt{n}}-2 \sqrt{n^a}\bigg) 
\weaklimit{n \to \infty}
F_{\text{TW}}.
\end{equation*}
\end{theorem}

\proof
We follow the lines of the proof of Theorem \ref{main}, 
emphasizing the points one should
be careful about or which need to be modified.

The construction of the upper and lower bounds for $L_{n,[n^a]}$ is the same 
as in Section \ref{SP}. 
In Lemma \ref{ac} and Lemma \ref{clemma}, the term $n^aE[\Delta_1]/n^b$ converges to 0 if
$a+1/r<b$, which leads to the constraint $a<6/7(1/2-1/r)$.

Next, we rewrite, as in Lemma \ref{limlaw},  
$$\lvert\sigma^{-1}S_{n,m}-Z_{\lambda n,m}\rvert\leq
2 \sum_{j=1}^m U^{(j)}_n + 2 \sum_{j=1}^m V^{(j)}_n$$
and prove convergence of each term separately.

\emph{Proof of the first convergence.}
Instead of Theorem \ref{KMTcoupling}, we use the combination of
results from \cite{KMT76} and \cite{MAJ76}, 
as stated in Proposition 2 of \cite{BM2005}, which allow us
to couple Brownian motions
$B^{(j)}$, $j=1,2,\ldots$ and random walks $R^{(j)}$, $j=1,2,\ldots$
so that
\begin{equation}
\label{KMT2}
 P\big(\max_{1\le i \le n} |B^{(j)}_i - R^{(j)}_i| > x\big) \le C n x^{-r},
\text{for all $n \in \N$, and all $x \in [n^{1/r},n^{1/2}]$. }
\end{equation}

Let $\epsilon=1/2$. We want to establish convergence of the same terms
 as in \eqref{1st}.
Applying, on the first term, the Markov inequality and
properties of the coupling with Brownian motion \eqref{KMT2} yield:
\begin{align*}
&P\big( \max_{1\le j \le [n^a]} U^{(j)}_n  \le n^{1/2}, 
  \sum_{j=1}^{[n^a]} U^{(j)}_n > \delta n^b\big)
  \le \frac{n^a}{\delta n^b}
   E\big[U_n^{(1)} \1(U_n^{(1)} \le n^{1/2})\big] \\
&\leq \frac{n^a}{\delta n^b}\left(n^{1/r}+
 \int_{n^{1/r}}^{n^{1/2}} P(U_n^{(1)}>x)dx\right)
  \leq \frac{n^a}{\delta n^b} \left( n^{1/r}+\int_{n^{1/r}}^{n^{1/2}}Cnx^{-r}dx\right)\\
& = \frac{n^a}{\delta n^b} \left( n^{1/r}-\frac{C}{r-1}n^{3/2-r/2}+\frac{C}{r-1}n^{1/r}\right)
  \leq \left(1+\frac{C}{r-1}\right)\frac{n^an^{1/r}}{\delta n^b}\rightarrow 0
\end{align*}
as $n\rightarrow\infty$. For the second term we use again the 
coupling properties \eqref{KMT2}:
\[
P\big(\max_{1\le j \le [n^a]} U^{(j)}_n > n^{1/2}\big)
\le n^a P\big(\max_{1\le i \le n} |B^{(1)}_i - R^{(1)}_i| > n^\epsilon\big)
\le n^a C n n^{-r/2} \to 0,
\]
as $n \to \infty$ because $a+1-r/2<0$ due to our new constraint on $a$.

\emph{Proof of the second convergence.}
Here, the only change is the replacement of Lemma \ref{convrate1}
by part of Theorem 6.12.1 in \cite{GUT}, in our notation:

If $1/2\leq q<1$, then for all $\epsilon>0$ it holds that
 $n^{qr-2}P(\max_{0\leq k\leq n}\lvert \Gamma_k-\lambda k \rvert>\epsilon n^q)\rightarrow 0$
as $n\rightarrow \infty$.

One can, in the same fashion as before, prove the following analogue of Corollary \ref{convrate2}, 
$$n^{qr-2}P\big(\sup_{0\leq t\leq n}\lvert \Phi(t)-\lambda t \rvert>\epsilon n^q\big)\rightarrow 0
\quad\text{as}\quad n\rightarrow \infty.$$

Likewise in Remark \ref{rem3/14}, condition $a<3/14$ is necessary.
Another constraint that occurs is $a\leq qr-2$ and it can be shown that 
this is satisfied if $a\leq (r-2)/(3r/7+1)$.
\qed

\begin{appendix}
\section{}
\begin{appxlem}
\label{max1}
Let $r\geq 1$ and $\{X_i,i\geq 1\}$ be a sequence of non-negative i.i.d.\ random variables 
such that $E  \vert X_1\vert ^r<\infty.$ Then,
$$\frac{1}{n^\frac{1}{r}}E\left[ \max_{1\leq i\leq n} X_i \right]\rightarrow 0 \text{\ as\ } 
 n\rightarrow \infty.$$
\end{appxlem}
\proof 
It suffices to prove the lemma for $r=1$, since then the 
result for $r>1$ follows easily by Jensen's inequality.
Let $M_n=\max_{1\leq i\leq n} X_i$ and $S_n=X_1+X_2+\dots+X_n$.
Borel-Cantelli lemma, together with the assumption $EX<\infty$, 
give $X_n/n\rightarrow 0$, almost surely, 
as $n \rightarrow \infty$. An easy computation shows that then also
$M_n/n \rightarrow 0$, almost surely, as $n\rightarrow \infty$. 

On the other hand, we know that 
$M_n\leq S_n$ for all $n$
and that $\{S_n/n,n\geq 1\}$ is uniformly integrable. Therefore, 
$M_n/n$ is also unformly integrable and 
$E M_n/n\rightarrow 0$ as $n\rightarrow \infty$ (\citealp[Thm. 5.4.5, Thm. 5.5.2]{GUT}).
\qed

The following lemma is an analogue of a result by \cite[Prop. 2]{LA98}, 
with slight improvement in the probability bound. 
Moreover, we specialize the lemma to our case.
\begin{appxlem}
\label{convrate1}
Let $1/2<q\leq1$ and $r<2-1/q$. Suppose that $X,X_1,X_2,\ldots$ are i.i.d.\ random 
variables such that $EX=\mu$ and 
$Ee^{\alpha \lvert X\rvert}<\infty$ for some $\alpha>0$, 
and set $S_n=\sum_{i=1}^{n}X_k$.
Then, for all $\epsilon>1$, 
\begin{equation*}
 \label{convrate11}
\exp\{\alpha n^{qr}\}P(\max_{1\leq k\leq n}\lvert S_k-k\mu\rvert>\epsilon n^q)\rightarrow 0
\end{equation*}
as $n\rightarrow \infty$.
\end{appxlem}

\proof
Without loss of generality we can assume that $\mu=0$.
We first show that, for all $\epsilon>1$, 
$\exp\{\alpha n^{qr}\}P(S_n>\epsilon n^q)\rightarrow 0$ as 
$n\rightarrow \infty$ .

For fixed $n\in\N$, define $X_k'=X_k\1(X_k\leq \epsilon n^q)$ and 
set $S_n'=\sum_{k=1}^{n}X_k'$.
Note that $EX'\leq0$ and that for $\alpha'$, where $\alpha'<\alpha$ and 
$\alpha'\epsilon^r>\alpha$, it holds that $EX^2e^{\alpha'\lvert X\rvert^r}<\infty$.
We can write
\begin{equation}
\label{blabla}
 P(S_n>\epsilon n^q)\leq P(S_n'>\epsilon n^q)+nP(X>\epsilon n^q).
\end{equation}
We first find a bound for the moment generating function for $X'$ in 
$\alpha'(\epsilon n^q)^{r-1}$:
\begin{align*}
 E e^{\alpha'(\epsilon n^q)^{r-1}X'}&\leq 
      1+\frac{1}{2}\alpha'^2(\epsilon n^q)^{2(r-1)}(EX'^2+
      EX'^2e^{\alpha'(\epsilon n^q)^{r-1}X'}\1\{X'>0\})\\
&\leq 1+\frac{1}{2}\alpha'^2(\epsilon n^q)^{2(r-1)}(EX^2+
      EX^2e^{\alpha'\lvert X\rvert^r})\\
&\leq 1+C n^{2q(r-1)}\leq \exp\{Cn^{-2q(1-r)}\},
\end{align*}
where for the first inequality we used $e^y\leq 1+\max\{1,e^y\}y^2/2$ and 
for the last one $1+y\leq e^y$. 
Now, using Markov's inequality and the bound above for the first term in \eqref{blabla} yields
\begin{align*}
 P(S_n'>\epsilon n^q)&=P(e^{\alpha'(\epsilon n^q)^{r-1}S_n'}>e^{\alpha'(\epsilon n^q)^{r}})
    \leq e^{-\alpha'(\epsilon n^q)^{r}}\left(Ee^{\alpha'(\epsilon n^q)^{r-1}X'}\right)^n\\
&\leq \exp\{-\alpha'(\epsilon n^q)^{r}+Cnn^{-2q(1-r)}\} .
\end{align*}
For the second term in \eqref{blabla}, again using Markov's inequality, we have
$$nP(X>\epsilon n^q)=nP(e^{\alpha X^r}>e^{\alpha (\epsilon n^q)^r})\leq 
    ne^{-\alpha (\epsilon n^q)^r}E(e^{\alpha \lvert X\rvert^r}).$$
Combining the two estimates finally establishes that
\begin{align*}
 \exp\{\alpha n^{qr}\}P(S_n>\epsilon n^q)&\leq 
   \exp\{-n^{qr}(\alpha'\epsilon^r -\alpha-Cn^{1-2q+qr})\}\\
& + n\exp\{-n^{qr}\alpha(\epsilon^r-1)\}Ee^{\alpha \lvert X\rvert^r}
\rightarrow 0
\end{align*}
as $n\rightarrow \infty$ because of the choice of $ r, \alpha'$ and $\epsilon$.

By symmetry, the same convergence rate holds also for $P(S_n<-\epsilon n^q)$.

Thus, the statement of the theorem follows using L\'evy inequality (\citealp[Thm. 3.7.2]{GUT})
and the fact that for $n$ large enough we can find $\epsilon'$ such that 
$1<\epsilon'<\epsilon-\sqrt{2\sigma^2}n^{1/2-q}$:
\begin{align*}
 P(\max_{1\leq k\leq  n} \lvert S_k \rvert>\epsilon n^q)&\leq 2P(\lvert S_n \rvert>\epsilon n^q-\sqrt{2n\sigma^2})
   =2P(\lvert S_n \rvert>n^q(\epsilon-\sqrt{2\sigma^2}n^{1/2-q}))\\
&\leq 2P(\lvert S_n \rvert>\epsilon' n^q) .
\end{align*}
\vskip-8mm\hfill$\Box$\vskip3mm

\begin{appxcor}
\label{convrate2}
Let $X,X_1,X_2,\ldots$ be positive, integer-valued i.i.d.\ random variables and suppose that 
the assumptions of Lemma \ref{convrate1} are satisfied.
Then, for the counting process $\Phi$, where $\Phi(t)=\max\{n:S_n\leq t\}$, it holds that
\begin{equation*}
 \label{convrate22}
\exp\{\alpha n^{qr}\}P(\sup_{0\leq t\leq n}\lvert \mu\Phi(t)-t\rvert>\epsilon n^q)\rightarrow 0
\text{\ for all $\epsilon>1$}.
\end{equation*}
\end{appxcor}
\proof
For the counting process we have $t-X_{\Phi(t)+1}\leq S_{\Phi(t)}\leq t$.
Therefore, we can write
\begin{align*}
&\{\sup_{0\leq t\leq n}\lvert \mu\Phi(t)-t\rvert>\epsilon n^q\}= \\
&=\{\sup_{0\leq t\leq n} (\mu\Phi(t)-t)>\epsilon n^q\}\cup
\{\inf_{0\leq t\leq n} (\mu\Phi(t)-t)<-\epsilon n^q\}\\
&\subset \{\sup_{0\leq t\leq n} (\mu\Phi(t)-S_{\Phi(t)})>\epsilon n^q\}\cup
\{\inf_{0\leq t\leq n} (\mu\Phi(t)-S_{\Phi(t)}-X_{\Phi(t)+1})<-\epsilon n^q\}\\
&\subset \{\sup_{0\leq t\leq n} (\mu\Phi(t)-S_{\Phi(t)})>\epsilon n^q\}
\cup \{\inf_{0\leq t\leq n} (\mu\Phi(t)-S_{\Phi(t)})<-\epsilon' n^q\} \\
& \ \quad  \cup \{\sup_{0\leq t\leq n} X_{\Phi(t)+1}>(\epsilon-\epsilon') n^q\}\\
&\subset \{\max_{1\leq k\leq  n} \lvert S_{k}-k\mu\rvert>\epsilon n^q\}
 \cup\{\max_{1\leq k\leq  n} \lvert S_{k}-k\mu\rvert>\epsilon' n^q\}\\
& \ \quad \cup  \{\max_{1\leq k\leq  n+1} X_{k}>(\epsilon-\epsilon') n^q\} ,
\end{align*}
where $1<\epsilon'<\epsilon$.
Thus,
\begin{align*}
 &\exp\{\alpha n^{qr}\}P(\sup_{0\leq t\leq n}\lvert \mu\Phi(t)-t\rvert>\epsilon n^q)
\leq \exp\{\alpha n^{qr}\}P(\max_{1\leq k\leq  n} \lvert S_{k}-k\mu\rvert>\epsilon n^q) \\
&+\exp\{\alpha n^{qr}\}P(\max_{1\leq k\leq  n} \lvert S_{k}-k\mu\rvert>\epsilon' n^q)
+(n+1)\exp\{\alpha n^{qr}\}P(X>(\epsilon-\epsilon') n^q).
\end{align*}
The first and the second term above converge to 0 as $n\rightarrow 0$ by Lemma \ref{convrate1}.
We prove convergence to 0 for the third term using Markov's inequality, 
$$(n+1)e^{\alpha n^{qr}}P(X>(\epsilon-\epsilon') n^q)\leq
(n+1)e^{\alpha (n^{qr}-(\epsilon-\epsilon') n^q)}Ee^{\alpha X}
\rightarrow 0\text{ as }n\rightarrow \infty.$$
\vskip-8mm\hfill$\Box$\vskip3mm

\end{appendix}

\section*{Acknowledgements}
We thank Allan Gut and Svante Janson for valuable comments.

\bibliographystyle{alea3}
\bibliography{TWslab_arxiv}

\end{document}